\documentclass{amsart}

\usepackage{hyperref}
\usepackage{comment} 

\usepackage[usenames]{color}
\usepackage[all]{xy}

\usepackage{amsfonts,amsmath,amssymb,amsthm,amscd,amsxtra}
\usepackage{enumerate,epsfig,verbatim}

\usepackage{enumerate,verbatim}
\usepackage{centernot}
\usepackage{todonotes}

\usepackage[T1]{fontenc}
\usepackage[usenames,dvipsnames]{pstricks}
\usepackage[mathscr]{eucal}
\usepackage[notcite,notref]{}
\usepackage{tikz}
\usetikzlibrary{matrix,quotes}
\usepackage[notcite, notref]{}

\usepackage[usenames,dvipsnames]{pstricks}
\usepackage{amsfonts,amsmath,amssymb,amsthm,amscd,amsxtra}
\usepackage{tikz-cd} 
\usepackage{enumerate,verbatim}
\usepackage{mathptmx}
\usepackage[notcite, notref]{}

\usepackage{amsfonts}
\usepackage[margin=1.4in]{geometry}
\usepackage{color}
\usepackage[notcite,notref]{}
\usepackage{url}
\usepackage{datetime}
\usepackage{mathtools}

\usepackage{amssymb}
\usepackage{amsmath}
\usepackage{amsfonts}
\usepackage{amsmath}
\usepackage{amsthm}
\usepackage{amssymb}
\usepackage{amscd}
\usepackage{amsfonts}
\usepackage{amsxtra} 

\usepackage{epsfig}
\usepackage{verbatim}

\SelectTips{cm}{}

\usepackage{amsfonts,amsmath,amssymb,amsthm,amscd,amsxtra}
\usepackage{enumerate,epsfig,verbatim}

\usepackage{amssymb,amstext,amsmath,amscd,amsthm,amsfonts,enumerate,graphicx,latexsym}
\usepackage[usenames]{color}

\usepackage{nicefrac}
\usepackage{array}

\usepackage{comment}
\usepackage{amssymb,amstext,amsmath,amscd,amsthm,amsfonts,enumerate,graphicx,latexsym,color,stmaryrd}
\usepackage{mathptmx}
\usepackage[all]{xy}
\usepackage{comment}

\newtheorem{thm}{Theorem}[section]
\newtheorem{prop}[thm]{Proposition}
\newtheorem{cor}[thm]{Corollary}

\theoremstyle{definition}
\newtheorem{chunk}[thm]{\hspace*{-1.065ex}\bf}

\newtheorem{dfn}[thm]{Definition}

\newtheorem{eg}[thm]{Example}

\newtheorem{que}[thm]{Question}

\newtheorem*{Claim}{Claim}

\newtheorem{rem}[thm]{Remark}

\theoremstyle{remark}

\numberwithin{equation}{thm}

\newcommand{\bZ}{\mathbb{Z}}

\newcommand{\fm}{\mathfrak{m}}

\newcommand{\bN}{\mathbb{N}}

\newcommand{\m}{\mathfrak{m}}
\newcommand{\fn}{\mathfrak{n}}
\newcommand{\q}{\mathfrak{q}}

\newcommand{\comp}{\mathbb{C}}

\def\fm{\mathfrak{m}}
\def\fp{\mathfrak{p}}

\def\Ext{\operatorname{\mathrm{Ext}}}

\def\Tor{\operatorname{\mathrm{Tor}}}
\def\q{\operatorname{\mathrm{q}}}

\def\syz{\Omega}

\def \cdim {\operatorname{CI-dim}}

\DeclareMathOperator{\Gdim}{\mathsf{G-dim}}
\DeclareMathOperator{\pd}{\mathsf{pd}}
\DeclareMathOperator{\cx}{\mathsf{cx}}

\DeclareMathOperator{\rGdim}{\operatorname{\mathsf{red-G-dim}}}

\DeclareMathOperator{\rpdim}{\operatorname{\mathsf{red-pd}}}
\newcommand{\rpm}{\raisebox{.2ex}{$\scriptstyle\pm$}}
\DeclareMathOperator{\Hdim}{\operatorname{\mathbf{\mathbb{I}}}}
\DeclareMathOperator{\rHdim}{\operatorname{\mathsf{red-\mathbb{I}}}}

\def \n {\mathfrak n} 
\def \q {\mathfrak q} 
 
\begin{document}

\title{}  
\title[On the reducing projective dimension of the residue field]{On the reducing projective dimension of the residue field}

\author{Olgur Celikbas}
\address{Olgur Celikbas\\ School of Mathematical and Data Sciences, West Virginia University, 
Morgantown, WV 26506 U.S.A}
\email{olgur.celikbas@math.wvu.edu}

\author{Souvik Dey}
\address{Department of Mathematics, University of Kansas, 405 Snow Hall, 1460 Jayhawk Blvd.,
Lawrence, KS 66045, U.S.A.}
\email{souvik@ku.edu}

\author{Toshinori Kobayashi}
\address{School of Science and Technology, Meiji University, 1-1-1 Higashi-Mita, Tama-ku, Kawasaki-shi, Kanagawa 214-8571, Japan}
\email{toshinorikobayashi@icloud.com}

\author{Hiroki Matsui}
\address{Hiroki Matsui\\Department of Mathematical Sciences\\
Faculty of Science and Technology\\
Tokushima University\\
2-1 Minamijyousanjima-cho, Tokushima 770-8506, Japan}
\email{hmatsui@tokushima-u.ac.jp}

\thanks{Toshinori Kobayashi was partly supported by JSPS Grant-in-Aid for JSPS Fellows 18J20660; Hiroki Matsui was partly supported by JSPS Grant-in-Aid for Early-Career Scientists 22K13894}

\maketitle

\begin{abstract} In this paper we are concerned with certain invariants of modules, called reducing invariants, which have been recently introduced and studied by Araya - Celikbas and Araya - Takahashi. We raise the question whether the residue field of each commutative Noetherian local ring has finite reducing projective dimension and obtain an affirmative answer for the question for a large class of local rings. Furthermore, we construct new examples of modules of infinite projective dimension that have finite reducing projective dimension, and study several fundamental properties of reducing dimensions, especially properties under local homomorphisms of local rings. 
\end{abstract}

\section{introduction} Throughout, $R$ denotes a local ring (commutative and Noetherian) with unique maximal ideal $\fm$ and residue field $k$, and each $R$-module is assumed to be finitely generated. For unexplained basic terminology, such as the definition of the Gorenstein dimension, we refer the reader to \cite{AuBr, BH, Gdimbook}.

The classical homological dimensions play a significant role in the characterizations of local rings. For example, a celabrated result of Auslander-Buchsbaum-Serre \cite[2.2.7]{BH} states that a local ring $R$ is regular if and only if the residue field $k$ has finite projective dimension. It is known that similar important characterizations of local rings via other homological dimensions of the residue field, such as via the Gorenstein dimension, also hold: a local ring $R$ is Gorenstein ring if and only if the Gorenstein dimension of the residue field $k$ is finite; see, \cite[4.20]{AuBr}.

Reducing versions of the classical homological dimensions have been recently introduced and studied by Araya and Celikbas \cite{CA}, and subsequently by Araya and Takahashi \cite{AT}; see Definition \ref{rdim} and Remark \ref{diff} for the details. These dimensions are finer than the classical dimensions: examples of modules that have finite reducing projective dimension but have infinite projective dimension are abundant; see \ref{s1} and Examples \ref{maineg}, \ref{e1} and \ref{e2}. In general, unlike the aforementioned characterizations via classical dimensions, the residue field of a non-Gorenstein local ring can have finite reducing Gorenstein, or finite reducing projective dimension; see, for example, \ref{egAC}. However, in addition to these examples, there are several noteworthy characterizations of local rings via reducing dimensions that make these invariants viable notations. One such result \cite{CDKM} states that a local ring $R$ is Gorenstein if and only if each $R$-module has finite reducing Gorenstein dimension; this generalizes the well-known fact that \cite[4.20]{AuBr} a local ring $R$ is Gorenstein if and only if each $R$-module has finite Gorenstein dimension.

It is interesting for us that the research that has been done so far concerning the reducing dimensions has not yet produced an example of a local ring whose residue field has infinite reducing projective dimension. Therefore, based on the examples and results we are aware of about these dimensions, it seems reasonable to us to ask the following:

\begin{que} \label{soruu} Does the residue field of each local ring have finite reducing projective dimension?
\end{que}

The main aim of this paper is to give support and obtain an affirmative answer for Question \ref{soruu} for a large class of local rings. More precisely, we prove:

\begin{thm} \label{mainthm} Let $(S, \fn)$ be a local ring such that the $\fn$-adic completion $\widehat S$ of $S$ equals $R/\underline{x}R$, where $(R, \fm)$ is a Cohen-Macaulay local ring of minimal multiplicity and $\{\underline{x} \}\subseteq  \fm$ is an $R$-regular sequence. If $k$ is infinite, then it has finite reducing projective dimension as an $S$-module.  \qed
\end{thm}   

We give a proof of Theorem \ref{mainthm} in section 3, where we also prove several preliminary results and study fundamantal properties of reducing dimensions, especially under local homomorphisms of local rings; see Theorem \ref{nzdse} and Remark \ref{nzdseremark}. It is worth noting that, if a module over a local ring as in Theorem \ref{mainthm} has finite Gorenstein dimension, then it does not need to have finite projective dimension; see Example \ref{exendo}. An interesting property of these rings is that they do not admit such modules as long as reducing dimensions are considered. For example, if $S$ is a non-Gorenstein complete local ring as in Theorem \ref{mainthm}, then we prove in Proposition \ref{endo} that each (finitely generated) $S$-module has finite reducing Gorenstein dimension if and only if it has finite reducing projective dimension. 

Another focus of this paper is to give new and nontrivial examples of modules of finite reducing projective dimension; we obtain such examples in section 2. One of these examples deserves to be mentioned here: we construct an example of a local ring where the reducing projective dimension of its residue field is two, and we determine a reducing projective dimension sequence of that residue field explicitly; see Example \ref{maineg}. 

The last section of this paper is devoted to the (involved) proof of Proposition \ref{213}; this proposition plays an important role in the proof of Theorem \ref{mainthm}. It also shows that, like the classical homological dimensions, the reducing dimension of a module is not affected by taking direct sum with a free module.




\section{Preliminaries, examples, and remarks}

In this section we give new examples of modules that have finite reducing, but infinite, projective dimension; see Examples \ref{maineg}, \ref{e1} and \ref{e2}. We also record several important facts that are needed for our argument; see \ref{obs}. We start by recalling the definition of reducing invariant of modules.

In the following, $\Hdim$ denotes a \emph{homological invariant} of $R$-modules, that is, $\Hdim$ denotes a function from the set of isomorphism classes of $R$-modules to the set $\bZ \cup \{ \rpm \infty \}$. Classical examples of such an invariant we use in this paper are the projective dimension $\Hdim=\pd$, and the Gorenstein dimension $\Hdim=\Gdim$.

\begin{dfn} \label{rdim} (\cite[2.5]{AT}; see also \cite[2.1]{CA}) Let $R$ be a local ring, $M$ be an $R$-module, and let $\Hdim$ denote a \emph{homological invariant} of $R$-modules.

We write $\rHdim_R(M)<\infty$ provided that there exist integers $r\geq 1$, $a_i\geq 1$, $b_i\geq 1$, $n_i\geq 0$, and short exact sequences of $R$-modules of the form 
\begin{equation}\tag{\ref{rdim}.1}
0 \to K_{i-1}^{{\oplus a_i}} \to K_{i} \to \Omega_R^{n_i}K_{i-1}^{{\oplus b_i}} \to 0,
\end{equation}
for each $i=1, \ldots r$, where $K_0=M$ and $\Hdim_R(K_r)<\infty$.  If a sequence as in (\ref{rdim}.1) exists, then we call $\{K_0,  \ldots, K_r\}$ a \emph{reducing $\Hdim$-sequence} of $M$. 

We set that the \emph{reducing invariant} $\rHdim_R(M)$ of $M$ is zero if and only if $\Hdim_R(M)<\infty$. Moreover, if $\Hdim_R(M)=\infty$, then we define:\begin{equation}\notag{}
\rHdim_R(M)=\inf\{ r\in \bN: \text{there is a reducing $\Hdim$-sequence }  K_0,  \ldots, K_r \text{ of }  M\}. 
\end{equation}
Therefore $0\leq \rHdim_R(M)\leq \infty$ for each $R$-module $M$.
\end{dfn}

We will give various examples of modules that have finite reducing homological dimension, but first we record several remarks. The first one, which follows directly from Definition \ref{rdim}, is used throughout the paper, for example, in the proofs of Proposition \ref{p1}, Proposition \ref{p2}, and Theorem \ref{nzdse}.

\begin{rem} \label{WH} Let $R$ be a local ring, $N$ be an $R$-module, and let $\Hdim$ denote a \emph{homological invariant} of $R$-modules. 
\begin{enumerate}[\rm(i)]
\item If $1\leq \rHdim_R(N)=r<\infty$, then there is an exact sequence $0\to N^{\oplus a} \to K \to \syz^n_R N^{\oplus b}\to 0$, where $a\geq 1$, $b\geq 1$, and $n\geq 0$ are integers, and $K$ is an $R$-module such that $\rHdim_R(K)=r-1$.
\item If $0\to N^{\oplus a} \to K \to \syz^n_RN^{\oplus b}\to 0$ is an exact sequence of $R$-modules, where $a\geq 1$, $b\geq 1$, and $n\geq 0$ are integers, then $\rHdim_R(N)\leq \rHdim_R(K)+1$. \qed
\end{enumerate}
\end{rem}

\begin{rem} \label{diff} The definition of the reducing dimension we use in this paper is taken from \cite{AT}, even though this notion was originally defined in \cite{CA}. The difference between the definitions of the reducing dimension given \cite{CA} and \cite{AT} is that \cite{AT} only requires the integers $n_i$ in Definition \ref{rdim} to be nonnegative, while \cite{CA} requires these numbers to be positive; see also Remark \ref{difference}. Hence, if a module has finite reducing invariant with respect to the definition given in \cite{CA}, then it has also finite reducing invariant with respect to the definition we use in this paper. \qed
\end{rem}

\begin{rem} The definition of reducing projective dimension \cite{CA} was initially originated from reducible complexity, a notion introduced by Bergh; see  \cite{Be, Berghredcx2} for the details. In general, modules having reducible complexity also have finite reducing projective dimension, but not vice versa. In other words, reducing projective dimension is a finer invariant than reducible complexity. One way to observe this fact is to use Theorem \ref{mainthm}: if $S$ is as in Theorem \ref{mainthm} and $S$ is not a complete intersection, then $k$ does not have finite complexity as an $S$-module \cite[2.3]{Gul}, and hence it does not have reducible complexity as an $S$-module \cite[2.1]{Be}, but $k$ has finite reducing projective dimension as an $S$-module by Theorem \ref{mainthm}. \qed
\end{rem}

One of the motivations of our work in this paper comes from the fact that the reducing projective of the residue field is finite over local rings with radical square zero:

\begin{eg}\label{egAC} Let $R$ be a local ring. Assume $R$ is not a field and $\fm^2 = 0$. Then it follows that $\pd_R(\Omega_R^ik)=\infty=\Gdim_R(\Omega_R^ik)$ and $\rpdim_R(\Omega_R^ik)=1=\rGdim_R(\Omega_R^ik)$ for all $i\geq 0$; see \cite[2.3]{CA}.
\end{eg}

\begin{eg} Let $R=\comp[\![x,y]\!]/(x^2,xy, y^2)$. Then  $\fm^2 = 0$ so that \ref{egAC} implies $\rpdim_R(k)=1$.\qed
\end{eg}

The ring in Example \ref{egAC} is zero-dimensional. Next is an example of a module -- over a local ring of dimension two -- that has finite reducing projective, but infinite projective, dimension.

\begin{eg} (\cite[2.7]{CA}) Let $R=\comp[\![x^3, x^2y, xy^2, y^3 ]\!]$ be the $3$rd Veronese subring of the ring $\comp[\![x, y]\!]$ and let $M=(x^2,xy, y^2)$. Then it follows that $\rpdim_R(M)=\rGdim_R(M)=1<\infty=\Gdim_R(M)$. In fact, one can check that there is an exact sequence of $R$-modules $0 \to M^{\oplus 2}  \to R^{\oplus 3} \to M \to 0$ so that $\{M ,R^{\oplus 3}\}$ gives a reducing $\pd$-sequence of $M$. \qed
\end{eg}

In general, unlike the classical homological dimensions, the reducing projective dimension of a nonzero module over a local ring, if finite, may not be bounded by the depth of the ring in question; see, for example, Example \ref{egAC}. In fact, we do not know whether or not there is a uniform upper bound for reducing projective dimension of modules, even over Gorenstein rings. On the other hand, a consequence of what we know from \cite{CDKM} yields a satisfactory answer over complete intersection rings, and also allows us to determine modules that have sufficiently large finite reducing projective dimension:

\begin{chunk} \label{s1} Let $R$ be a local complete intersection ring of codimension $c$, for example, one can pick the ring $R=k[\![x_1, \ldots, x_c, y_1, \ldots, y_c]\!]/(x_1y_1, \ldots, x_cy_c)$. If $M$ is an $R$-module, then it follows that $\rpdim_R(M)\leq c$ since $\rpdim_R(M)=\cx_R(M)$, where $\cx_R(M)$ is the complexity of $M$; see \cite{CDKM} for the details. So, for each given $r$ with $0\leq r \leq c$, one can use \cite[5.7]{AGP} and construct an $R$-module $M$ such that $\rpdim_R(M)=r$. \qed

\end{chunk}

We proceed to give an example of a module that has reducing projective dimension equal two and subsequently determine a reducing sequence of it explicitly; see Example \ref{maineg}. First we record some auxiliary results that are used for the argument of the example, as well as in section 3.

\begin{chunk}\label{obs} Let $(R,\m)\to (S,\n)$ be a homomorphism of local rings.%
\begin{enumerate}[\rm(i)]
\item If $M$ is an $R$-module and $\Tor^R_{i}(M,S)=0$ for each $i\geq 1$, then it follows that $\pd_S(M\otimes_R S)=\pd_R(M)$ and $S\otimes_R \syz^n_R M\cong \syz^n_S(M\otimes_R S)$ for all $n\ge 0$; see  \cite[1.2.3]{Av2}.   
\item Assume $S$ has finite flat dimension over $R$. Let $M$ be an $R$-modules such that $\Gdim_R(M)=n<\infty$ and $\Tor^R_{i}(M,S)=0$ for all $i=1, \ldots, n$. Then \cite[5.4.4]{Gdimbook} implies that $\Tor^R_{i}(M,S)=0$ for all $i\geq 1$. Note that $M$ is reflexive as an $R$-complex; see, for example,  \cite[2.7]{Yass}. Therefore, we conclude from \cite[5.10]{Larsthesis} that $\Gdim_S(M\otimes_R S)=n$. 
\item If $S$ is a finitely generated $R$-module such that $\pd_R(S)=n<\infty$ and $N$ is a finitely generated $S$-module, then there is some free $R$-module $F$ such that $\syz^{ a}_R\syz^{ b}_S N\cong F \oplus \syz^{ a+b}_RN$ 
for each $a\ge n$ and $b\ge 0$; this fact extends \cite[4.2]{TN} and can be proved similarly. \qed
\end{enumerate}
\end{chunk}   

The next observation is used for Example \ref{maineg} and also in the proof of Proposition \ref{p2}.

\begin{chunk} \label{ara} Let $R$ be a local ring and let $M$ be a nonzero $R$-module. Assume there is a non zero-divisor $x\in \fm$ on $R$ such that $xM=0$. Then there is a short exact sequence of $R/xR$-modules:
\begin{equation}\tag{\ref{ara}.1}
0 \to M \to \Omega_R M \otimes_R (R/xR) \to \Omega_{R/xR}(M) \to 0. 
\end{equation}  

To establish the sequence in (\ref{ara}.1), we consider a syzygy sequence of $M$, namely an exact sequence $0 \to \Omega_R M \to F \to M \to 0$, where $F$ is a free $R$-module. Then we obtain, by tensoring the syzygy sequence with $R/xR$, an exact sequence of $R$-modules of the form
\begin{equation}\tag{\ref{ara}.2}
0 \to \Tor_1^R(M,R/xR) \to \Omega_R M \otimes_R (R/xR) \to F\otimes_R(R/xR) \to M/xM \to 0.
\end{equation}
Note that $\Tor_1^R(M,R/xR) \cong M$ since $x$ is a non zero-divisor on $R$ and $xM=0$. Hence (\ref{ara}.2) yields the exact sequence in (\ref{ara}.1). \qed
\end{chunk}

\begin{eg} \label{maineg} Let $R=S/aS$, where $(S, \fn)$ is a one-dimensional local hypersurface ring and $a\in \fn^2$ is a nonzero divisor on $S$. For example, we can set $R=\comp[\![x,y]\!]/(x^2, y^2)$, $S=\comp[\![x,y]\!]/(x^2)$, and $a=y^2$. Then $R$ is an Artinian complete intersection ring of codimension two so that $\rpdim_R(k)=2$; see \ref{s1}. Now we proceed and find a reducing $\pd$-sequence of $k$ over $R$. 

Let us first note, by making use of (\ref{ara}.1) with the $S$-module $k$ over $S$, we have a short exact sequence of $R$-modules of the form:
\begin{equation}\tag{\ref{maineg}.1}
0 \to k \to R \otimes_{S} \fn  \to \Omega_R k \to 0.
\end{equation}
As we know $\rpdim_R(k)=2$, it follows that $\rpdim_R(R \otimes_{S} \fn)=1$; see Remark \ref{WH}(ii). So it is enough to find a reducing $\pd$-sequence of  $R \otimes_{S} \fn$ over $R$.

In general, if $T$ is a local ring and $X$ and $Y$ are (finitely generated) $T$-modules such that $\pd_T(X)=1$ and $Y$ is torsion-free, then $\Tor_i^T(X,Y)=0$ for al $i\geq 1$; see, for example, \cite[2.7]{CET}. So, as $\pd_S(R)=1$, and $\fn$ and $\Omega_S \fn$ are torsion-free $S$-modules, we have:
\begin{equation}\tag{\ref{maineg}.2}
\Tor_i^S(\Omega_S \fn,R)=0=\Tor_i^S(\fn,R) \text{ for all } i\geq 1.
\end{equation}

Note that, as $S$ is a one-dimensional hypersurface ring and $\fn$ is a maximal Cohen-Macaulay $S$-module with no free summand, we have $\fn \cong \Omega^2_S \fn$. Moreover, since $\fn$ is generated by two elements, there is a short exact sequence of $S$-modules of the form:
\begin{equation}\tag{\ref{maineg}.3}
0 \to \fn \to S^{\oplus 2} \to \Omega_S \fn\to 0.
\end{equation}
We obtain, by using (\ref{maineg}.2) and tensoring (\ref{maineg}.3) with $R$ over $S$, the exact sequence of $R$-modules:
\begin{equation}\tag{\ref{maineg}.4}
0 \to   R \otimes_S \fn  \to R^{\oplus 2} \to R \otimes_S \Omega_S \fn  \to 0.
\end{equation}

Next note that $R\otimes_S\Omega_S \fn  \cong \Omega_R (\fn \otimes_S R)$; see \ref{obs}(i) and (\ref{maineg}.2). Therefore, (\ref{maineg}.1) and (\ref{maineg}.4) show that a reducing $\pd$-sequence of $k$ over $R$ is $\{k, R \otimes_S  \fn, R^{\oplus 2} \}$. \qed
\end{eg} 

Recall that an $\fm$-primary ideal $I$ of a local ring $R$ is called \emph{Ulrich} if  $I/I^2$ is a free $R/I$-module and $I^2 = \mathfrak{q}I$ for some parameter ideal $\mathfrak{q}$ of $R$, where $\mathfrak{q}$ is a reduction of $I$; see \cite[1.1]{GotoUlrich} for the details. Next we point out that Ulrich ideals, when exist, provide nontrivial examples of modules of finite reducing projective dimension.

\begin{chunk}\label{min} Let $R$ be a Cohen-Macaulay local ring and let $I$ be an Ulrich ideal of $R$ that is not a parameter ideal. Then it follows that $\pd_R(R/I)=\infty$ and  $\rpdim_R(R/I) = 1$. 

To observe this, note,  if $\rpdim_R(R/I) = 0$, that is, $\pd_R(R/I)<\infty$, then, as $I/I^2$ is a free $R/I$-module, and so \cite[2.2.8]{BH} implies that $I$ is a parameter ideal. Hence we conclude $\pd_R(R/I)=\infty$. Next we consider a parameter ideal $\mathfrak{q}$ of $R$ such that $I^2 = \mathfrak{q}I$ and $\mathfrak{q} \subsetneqq I \subseteq \fm$. Then, by \cite[2.3]{GotoUlrich}(2)(c), it follows that $I/\mathfrak{q}$ is a free $R/I$-module. So there is an exact sequence of $R$-modules $0 \to (R/I)^{\oplus a} \to R/\mathfrak{q} \to R/I \to 0$ for some $a\geq 1$, which shows that $\rpdim_R(R/I) \le 1$, that is, $\rpdim_R(R/I) = 1$.
\end{chunk}

\begin{rem} \label{difference} It is worth mentioning that, if we adopt the definition of the reducing dimension given in \cite{CA}, then the argument of \ref{min} does not necessarily yield that $\rpdim_R(R/I) = 1$.
\end{rem}

\begin{eg} \label{e1} Let $R= \comp[\![x,y,z]\!]/(x^3-y^2, z^2-x^2y)$ and let $I=(x,y)$. Then $R/I\cong \comp[\![z]\!]/(z^2)$ and $I$ is an Ulrich ideal of $R$ which is not a parameter ideal; see \cite[2.7(1)]{GotoUlrich}. Therefore, by \ref{min}, it follows that $\pd_R(R/I)=\infty$ and  $\rpdim_R(R/I) = 1$; see also \cite[2.7, 6.7 and 6.8]{GotoUlrich} for similar examples. \qed
\end{eg}

Recall that a Cohen-Macaulay local ring $R$ is said to have \emph{minimal multiplicity} if the codimension of $R$ is one less than the multiplicity of $R$; see \cite[4.5.14]{BH}.  In the following we obtain an extension of the fact recorded in Example \ref{egAC}:

\begin{chunk} \label{mincor} Let $R$ be a non-regular Cohen-Macaulay local ring such that $R$ has minimal multiplicity and $|k|=\infty$. Then it follows that $\fm^2=\mathfrak{q}\fm$ for some parameter ideal $\mathfrak{q}$ that is a minimal reduction of $\fm$; see, for example, \cite{Sally}. Then it follows that $\fm$ is an Ulrich ideal of $R$ that is not a parameter ideal. Therefore \ref{min} yields that $\pd_R(k)=\infty$ and $\rpdim_R(k) = 1$.  
\end{chunk}

\begin{eg} \label{e2} Let $R=\comp[\![x,y]\!]/(x,y)^2$, or $R=\comp[\![t^3, t^4, t^5]\!]$, or $R=\comp[\![t^4, t^5, t^6, t^7]\!]$. Then $R$ is a non-regular Cohen-Macaulay local ring with minimal multiplicity so that \ref{mincor} implies that $\pd_R(k)=\infty$ and $\rpdim_R(k) = 1$. \qed
\end{eg}

\begin{rem} We now have several examples of modules that have finite reducing projective dimension. Let us also mention that modules of infinite reducing projective dimension also exist: Jorgensen and {\c{S}}ega \cite[1.7]{JS} constructed a local Artinian ring $T$ and a $T$-module $X$ of infinite Gorenstein dimension such that $\Ext^i_T(X,T) = 0$ for all $i\geq 1$. On the other hand, Araya and Celikbas \cite[1.3]{CA} proved that, if $R$ is a local ring and $M$ is an $R$-module such that $\rGdim_R(M)<\infty$ and $\Ext^i_R(M,R) = 0$ for all $i\gg 0$, then $M$ has finite Gorenstein dimension. Therefore the aforementioned module $X$ constructed by Jorgensen and {\c{S}}ega has infinite reducing Gorenstein dimension, and hence has infinite reducing projective dimension, over the ring $T$. \qed
\end{rem}

We prove Theorem \ref{mainthm} in the next section; we should note that Theorem \ref{mainthm}, besides giving an affirmative answer for Question \ref{soruu} for a large class of  local rings, also yields an extension of the observation stated in 2.12.

\section{Proof of the main theorem} 

In this section we give a proof of Theorem \ref{mainthm}; see the paragraph following Proposition \ref{p2}. Along the way we obtain several new results and study further properties of reducing dimensions, especially properties under local homomorphisms of local rings.

Throughout, given a ring $R$, $\Hdim$ denotes a homological invariant of $R$-modules; see Definition \ref{rdim}. Our first result yields a generalization of \cite[2.9]{ACCK}. 

\begin{prop} \label{p1} Let $(R,\m)\to (S,\n)$ be a local homomorphism of local rings. Assume one has that $\Hdim_S(M\otimes_R S)\le \Hdim_R(M)$ for each $R$-module $M$ for which $\Tor^R_{i}(M,S)=0$ for all $i\geq 1$. If $N$ is an $R$-module such that $\Tor^R_{i}(N,S)=0$ for all $i\geq 1$, then $\rHdim_S(N\otimes_R S)\leq \rHdim_R(N)$.
\end{prop}

\begin{proof} Let $N$ be an $R$-module such that $\rHdim_R(N)=r<\infty$. We proceed by induction on $r$ to prove $\rHdim_S(N\otimes_R S)\leq r$. We may assume $N\otimes_R S\neq 0$.

Assume $r=0$. Then, by Definition \ref{rdim}, we have that $\Hdim_R(N)<\infty$. This implies, by the hypothesis, that $\Hdim_S(N\otimes_R S)\leq  \Hdim_R(N)<\infty$, and hence $\rHdim_S(N\otimes_R S)=0=r$.

Next we assume $r\geq 1$. Then, by Remark \ref{WH}(i), there exists a short exact sequence of $R$-modules $0\to N^{\oplus a} \to K \to \syz^c_R N^{\oplus b}\to 0$, where $\rHdim_R(K)=r-1$, and $a\geq 1$, $b\geq 1$, and $c\geq 0$ are integers. Note, since  $\Tor^R_{i}(N,S)=0$ for all $i\geq 1$, we have the vanishing of $\Tor^R_{i}(K,S)$ for all $i\geq 1$. Hence the induction hypothesis on $r$ yields: 
\begin{equation}\tag{\ref{p1}.1}
\rHdim_S(K\otimes_R S)\leq \rHdim_R(K)= r-1.
\end{equation}
Once again we make use of the hypothesis that $\Tor^R_{i}(N,S)$ vanishes for all $i\geq 1$, tensor the short exact sequence $0\to N^{\oplus a} \to K \to \syz^c_R N^{\oplus b}\to 0$ with $S$ over $R$, and obtain the short exact sequence of $S$-modules $0\to (N\otimes_R S)^{\oplus a} \to K\otimes_R S \to \syz^c_S(N\otimes_R S)^{\oplus b}\to 0 $; see also \ref{obs}(i). So, by Remark \ref{WH}(ii), we deduce:
\begin{equation}\tag{\ref{p1}.2}
\rHdim_S(N\otimes_R S)\le \rHdim_S(K\otimes_R S)+1.
\end{equation}
Consequently (\ref{p1}.1) and (\ref{p1}.2) show that $\rHdim_S(N\otimes_R S)\leq r=\rHdim_R(N)$.
\end{proof}

\begin{cor} \label{p1c} Let $(R,\m)\to (S,\n)$ be a local homomorphism of local rings and let $N$ be an $R$-module such that $\Tor^R_{i}(N,S)=0$ for all $i\geq 1$. Then the following hold:
\begin{enumerate}[\rm(i)]
\item $\rpdim_S(N\otimes_R S)\le \rpdim_R(N)$.
\item If $S$ has finite flat dimension over $R$, then  $\rGdim_S(N\otimes_R S)\le \rGdim_R(N)$.
\end{enumerate}
\end{cor}

\begin{proof} The first and the second conclusion, in view of Proposition \ref{p1}, follow from \ref{obs}(i) and \ref{obs}(ii), respectively.
\end{proof}

\begin{cor} \label{p1c2} Let $R$ be a local ring and let $x\in \fm$ be a non zero-divisor on $R$. Assume one has that $\Hdim_{R/xR}(N/xN)\leq \Hdim_R(N)$ for each $R$-module $N$ such that $x$ is a non zero-divisor on $N$. If $M$ is an $R$-module and $x$ is a non zero-divisor on $M$, then $\rHdim_{R/xR}(M/xM)\leq \rHdim_R(M)$.
\end{cor}

\begin{proof} The result follows from Proposition \ref{p1} by using the ring homomorphism $R \to S=R/xR$.
\end{proof}

The following special case of Corollary \ref{p1c2} is worth recording separately:

\begin{cor} \label{p1c3} Let $R$ be a local ring and let $M$ be an $R$-module. If $x \in \fm$ is a non zero-divisor on $R$ and $M$, then $\rpdim_{R/xR}(M/xM)\leq \rpdim_R(M)$ and $\rGdim_{R/xR}(M/xM)\leq \rGdim_R(M)$. \qed
\end{cor}

Next we point out that the inequality in Corollary \ref{p1c}(ii) can be strict; we give an example of a local homomorphism of local rings $R \to S$ and an $R$-module $N$ such that $S$ has finite flat dimension over $R$, $\Tor^R_{i}(N,S)=0$ for all $i\geq 1$, and $\rGdim_S(N\otimes_R S)<\infty=\rGdim_R(N)$.  

\begin{eg} \label{newex} Let $R=k[\![t^3, t^7, t^8]\!]$ and let $N=R+Rt+Rt^5$. Then $R$ is a one-dimensional domain and $N$ is a maximal Cohen-Macaulay $R$-module. 

Let $S=R/\q$, where $\q=(t^3)$. Then $R\cong k[\![x,y,z]\!]/(x^5-yz, y^2-zx^2, z^2-x^3y)$ and $S \cong k[\![y,z]\!]/(y,z)^2$. Moreover, we have that $\q N=Rt^3+Rt^4+Rt^8=\m N$ and hence $N/\q N \cong k^{\oplus 3}$. Therefore it follows that $\rpdim_{S}(N\otimes_RS)=1=\rGdim_{S}(N\otimes_RS)$; see Example \ref{egAC}.

Note that $N$ is reflexive if and only if $N=(R:_{Q(R)}(R:_{Q(R)}N))$; see, for example, \cite[2.4]{KoTa}. One can check
the equalities $\big(R:_{Q(R)}(R:_{Q(R)}N)\big)=\big(R:_{Q(R)}(Rt^6+Rt^7+Rt^8)\big)=R+Rt+Rt^2$ hold and conclude that $N$ is not reflexive.

As $R$ is a one-dimensional domain, an $R$-module is reflexive if and only if it is the syzygy of a maximal Cohen-Macaulay $R$-module; see \cite[2.4]{IW}. Hence $N$ is not the syzygy of a maximal Cohen-Macaulay $R$-module. It is proved in \cite{CDKM} that, if a maximal Cohen-Macaulay module over a local ring has finite reducing Gorenstein dimension, then it a syzygy of a maximal Cohen-Macaulay module. Consequently, we conclude that $\rGdim_R(N)=\infty$. Finally we note, since $\pd_R(S)=1$ and $t^3$ is a non zero-divisor on $N$, that $\Tor^R_{i}(N,S)=0$ for all $i\geq 1$. \qed
\end{eg}

\begin{chunk} \label{setup} Given a local ring $R$, we consider the following conditions for the invariant  $\Hdim$ in question:
\begin{enumerate}[\rm(i)]
\item $\Hdim_R(M)\le \pd_R(M)$ for each $R$-module $M$. 
\item $\Hdim_{R}(\syz_{R}{M})\le \Hdim_{R}(M)$ for each $R$-module $M$.  
\item $\Hdim_R(M)\leq \Hdim_R(N)$ if $0 \to F \to N \to M \to 0$ is an exact sequence of $R$-modules, where $F$ is free.
\end{enumerate}
\end{chunk}


\begin{chunk} \label{conc} If $\Hdim$ equals the projective dimension $\pd$, or the Gorenstein dimension $\Gdim$, then it satisfies all the conditions stated in \ref{setup}; see, for example, \cite[1.2.9, 1.2.10]{Gdimbook} for the case where $\Hdim=\Gdim$. \qed
\end{chunk}

We need the following result for the proof of Theorem \ref{nzdse}.

\begin{prop} \label{213} Let $R$ be a local ring and let $\Hdim_R$ be a homological invariant of $R$-modules which satisfies the condition (iii) of \ref{setup}. Then it follows that $\rHdim_R(M\oplus F)=\rHdim_R(M)$ for each $R$-module $M$ and each free $R$-module $F$. \qed
\end{prop}

The conclusion of Proposition \ref{213} seems natural, but its proof is involved; therefore we defer the proof of the lemma until the next section so it does not interfere with the flow of the paper. 

\begin{prop} \label{p2} Let $R$ be a local ring and let $S=R/\underline{x}R$, where $\underline{x}=x_1,\ldots, x_n \subseteq \fm$ is an $R$-regular sequence. Assume the following hold:
\begin{enumerate}[\rm(i)]
\item $\Hdim$ satisfies the condition (iii) stated in \ref{setup}.
\item If $N$ is an $R$-module such that $x_1$ is a non zero-divisor on $N$, then $\Hdim_{R/x_1R}(N/x_1N)\leq \Hdim_R(N)$. 
\end{enumerate}
If $M$ is a (finitely generated) $S$-module, then it follows that $\rHdim_{S}(M)\le \rHdim_R(\syz^n_RM)+n$.
\end{prop}

\begin{proof} We proceed by induction on $n$. Assume first $n=1$ and set $x_1=x$. Then, since $x$ is a non zero-divisor on $\syz_R M$, it follows from Corollary \ref{p1c2} that $\rHdim_{S} \left(\dfrac{\syz_R M}{x\cdot \syz_RM} \right) \le \rHdim_R \left(\syz_R M\right)$. 

Note, by (\ref{ara}.1), we have the following short exact sequence of $S$-modules:
\begin{equation}\tag{\ref{p2}.1}
0 \to M \to \dfrac{\syz_R M}{x\cdot \syz_RM}  \to \Omega_{S}(M) \to 0. 
\end{equation}
Therefore Remark \ref{WH}(ii) applied to (\ref{p2}.1) gives $\rHdim_{S}(M)\leq \rHdim_{S} \left(\dfrac{\syz_R M}{x\cdot \syz_RM} \right) +1$. This establishes the base case $n=1$ of the induction. 

Next we assume $n\geq 2$ and set $T=R/\underline{x'}R$, where $\underline{x'}=x_1,\ldots,x_{n-1}$. Then $x_n$ is a non zero-divisor on $T$ and $S=T/x_nT$. So, by the case where $n=1$, we have:
\begin{equation}\tag{\ref{p2}.2}
\rHdim_{S}(M)\le \rHdim_T(\syz_T M)+1.
\end{equation}
We make use of the induction hypothesis with the module $\syz_T M$ and conclude that:
\begin{equation}\tag{\ref{p2}.3}
\rHdim_T(\syz_T M)\le \rHdim_R(\syz^{n-1}_{R}\syz_{T}M)+n-1. 
\end{equation}
Hence (\ref{p2}.2) and (\ref{p2}.3) yield:
\begin{equation}\tag{\ref{p2}.4}
\rHdim_{S}(M)\le \rHdim_R(\syz^{n-1}_{R}\syz_{T}M)+n. 
\end{equation}
Next, note that, $\syz^{n-1}_{R}\syz_{T}M \cong \syz^{n}_RM\oplus F$ for some free $R$-module $F$; see \ref{obs}(iii). Thus the induction argument is now complete from (\ref{p2}.4) and Proposition \ref{213}.
\end{proof}


We need the following observation for the proof of Theorem \ref{nzdse}.

\begin{chunk} \label{c26} Let $R$ be a ring (not necessarily local and Noetherian) and let $0 \to L \to M \xrightarrow[]{p} N\oplus F \to 0$ be a short exact sequence of $R$-modules, where $F$ is free. Then we have the following pullback diagram, that is, a commutative diagram with exact rows and columns:
\[
\begin{tikzcd}
 & & 0 \ar[d] & 0 \ar[d] & \\
0 \ar[r] & L \ar[r] \ar[d,equal] & M' \ar[r] \ar[d] & N \ar[r] \ar[d] & 0 \\
0 \ar[r] & L \ar[r] & M \ar[r,"p"] \ar[d] & N \oplus F \ar[r] \ar[d] & 0\\
 & & F \ar[r,equal] \ar[d] & F \ar[d] &\\
 & & 0 & 0 &
\end{tikzcd}
\]
This diagram yields a short exact sequence of $R$-modules $0 \to L \to M' \to N \to 0$, where $M\cong M'\oplus F$. \qed
\end{chunk}


The next theorem plays a key role for the proof of Theorem \ref{mainthm}. 

\begin{thm}\label{nzdse}  Let $(R,\m)\to (S,\n)$ be a local ring homomorphism, where $S$ is a finitely generated $R$-module and $\pd_R(S)\le n<\infty$ for some integer $n$. Assume $\Hdim$ satisfies all the conditions stated in \ref{setup}. 
Assume further that $\Hdim_R(N)<\infty$ whenever $N$ is a (finitely generated) $S$-module such that $\Hdim_S(N)<\infty$. Then $\rHdim_R(\syz^n_R N)\le \rHdim_S (N)$ for each finitely generated $S$-module $N$. 
\end{thm}  


\begin{proof} Let $N$ be a nonzero finitely generated $S$-module such that $\rHdim_S (N)=r<\infty$. We proceed by induction on $r$ to show that $\rHdim_R(\syz^n_R N)\le r$.

Assume first $r=0$. Then $\rHdim_S(N)=0$, i.e., $\Hdim_S(N)<\infty$. Hence, by the hypotheses, we conclude that $\Hdim_{R}(\syz_{R}{N})\le \Hdim_{R}(N)<\infty$. This shows that $\rHdim_R(\syz^n_R M)=0=r$.

Next assume $r\geq 1$. Then it follows from Definition \ref{rdim} that there exists a short exact sequence of (finitely generated) $S$-modules
\begin{equation}\tag{\ref{nzdse}.1}
0\to N^{\oplus a} \to K \to \syz^c_S N^{\oplus b}\to 0, 
\end{equation}
where $a\geq 1$, $b\geq 1$ and $c\geq 0$ are integers, and
$\rHdim_S(K)= \rHdim_S(N)-1=r-1$. So, by the induction hypothesis, we have: 
\begin{equation}\tag{\ref{nzdse}.2}
\rHdim_R(\syz^n_R K)\le \rHdim_S (K)=r-1.
\end{equation}

We obtain, by applying $\syz^n_R(-)$ to (\ref{nzdse}.1), the following exact sequence of $R$-modules:
\begin{equation}\tag{\ref{nzdse}.3}
0\to \syz^n_RN^{\oplus a} \to \syz^n_RK\oplus G \to \syz^{n}_R \syz^c_S N^{\oplus b} \to 0,
\end{equation}
where $G$ is a free $R$-module; see \cite[2.2(1)]{DaoTak}.

Note that $\syz^{n}_R \syz^c_S N^{\oplus b} \cong \syz^{n+c}_R N^{\oplus b}\oplus F$ for some free $R$-module $F$; see \ref{obs}(iii).
Then, by using \ref{c26}, we obtain the following exact sequence of $R$-modules:
\begin{equation}\tag{\ref{nzdse}.4}
0\to \syz^n_RN^{\oplus a} \to K' \to \syz^{n+c}_R N^{\oplus b} \to 0,
\end{equation}
where $\syz^n_RK\oplus G\cong K'\oplus F$.
Now we observe the following (in)equalities which complete the proof:
\begin{align}\notag{} 
\rHdim_R(\syz^n_R N) \leq & 
\rHdim_R(K')+1 \\ = & \tag{\ref{nzdse}.5} \rHdim_R(\syz^n_RK)+1 \\ \notag{} \leq &(r-1)+1=r. 
\end{align}
Here, in (\ref{nzdse}.5), the first inequality follows from (\ref{nzdse}.4) and Remark \ref{WH}(ii), and the second one follows from (\ref{nzdse}.2); moreover, the first equality is due to Proposition \ref{213}. 
\end{proof}

\begin{rem} \label{nzdseremark} Let $(R,\m)\to (S,\n)$ be a local ring homomorphism, where $S$ is a finitely generated $R$-module and $\pd_R(S)<\infty$. Then the hypothesis on the ring map considered in Theorem \ref{nzdse}, namely the condition that $\Hdim_R(N)<\infty$ whenever $N$ is a finitely generated $S$-module such that $\Hdim_S(N)<\infty$, holds when $\Hdim$ equals the projective dimension. This condition also holds for certain ring homomorphisms when $\Hdim$ equals the Gorenstein dimension. For example, if the closed fibre of the ring map $(R,\m)\to (S,\n)$ is Gorenstein, then it follows that $\Gdim_S(N)<\infty$ if and only if $\Gdim_R(N)<\infty$ for each finitely generated $S$-module $N$; see \cite[4.2]{LGH}  and \cite[7.11]{AFGD}. 
Therefore, if we consider the natural surjection $R \to S=R/(\underline{x})$ for some $R$-regular sequence $\underline{x} \subseteq \fm$, then $\Gdim_S(N)<\infty$ if and only if $\Gdim_R(N)<\infty$ for each finitely generated $S$-module $N$. \qed
\end{rem}

Next is a consequence of Theorem \ref{nzdse}; this result, for the case where $\Hdim$ is the projective dimension, or the Gorenstein dimension of modules, can also be deduced from \cite{CA}. However, as mentioned in Remark \ref{diff}, the definition given in \cite{CA} is slightly more restrictive than the definition of reducing dimensions we adopt in this paper; see Definition \ref{rdim}.

\begin{cor} \label{nzdsec1} Let $R$ be a local ring. Assume $\Hdim$ satisfies all the conditions stated in \ref{setup}. Then it follows that $\rHdim_R(\syz^i_R M)\le \rHdim_R (M)$ for each $R$-module $M$ and for each $i\geq 0$.
\end{cor}

\begin{proof} The corollary follows from Theorem \ref{nzdse} by setting $R=S$.
\end{proof}

We need a few more results to prove Theorem \ref{mainthm}; in the following $\widehat{(-)}$ denotes the completion functor.

\begin{chunk} \label{comple0} Let $R$ be a local ring. Assume one has that $\Hdim_{\widehat R}(\widehat M)=\Hdim_R(M)$ for each $R$-module $M$. If $N$ is an $R$-module, then it follows from Definition \ref{rdim} that $\rHdim_{\widehat R}(\widehat N)\leq \rHdim_R(N)$. \qed
\end{chunk}

We can show, under the hypothesis of \ref{comple0}, that $\rHdim_{\widehat R}(\widehat N)= \rHdim_R(N)$ if $N$ is an $R$-module which is locally free on the punctured spectrum of $R$, that is, if $N_{\fp}$ is free over $R_{\fp}$ for all non-maximal prime ideals $\fp$ of $R$.


\begin{prop}\label{comple} Let $R$ be a local ring. Assume one has that $\Hdim_{\widehat R}(\widehat M)=\Hdim_R(M)$ for each $R$-module $M$. If $N$ is an $R$-module that is locally free on punctured spectrum of $R$, then $\rHdim_{\widehat R}(\widehat N)=\rHdim_R(N)$.
\end{prop}

\begin{proof} Let $N$ be an $R$-module which is locally free on punctured spectrum of $R$. Then, in view of \ref{comple0}, it is enough to prove $\rHdim_R(N)\le \rHdim_{\widehat R}(\widehat N)$. Notice we may assume $\rHdim_{\widehat R}(\widehat N)<\infty$ as otherwise there is nothing to prove. We set $r=\rHdim_{\widehat R}(\widehat N)$ and proceed by induction on $r$.

If $r=0$, then we have $\Hdim_{\widehat R}(\widehat N)<\infty$; see Definition \ref{rdim}. Hence our assumption yields $\Hdim_R (N)=\Hdim_{\widehat R}(\widehat N)$, and so shows that $\rHdim_R (N)=0=\rHdim_{\widehat R}(\widehat N)$.

Now let $r\ge 1$. Then there exists a short exact sequence of $\widehat R$-modules: $0\to \widehat {N^{\oplus a}} \to  K \to \widehat{\syz^c_{ R} N^{\oplus b}}\to 0$, where $\rHdim_{\widehat R}( K)=r-1$, and $a\geq 1$, $b\geq 1$, and $c\geq 0$ are integers; see Remark \ref{WH}(i). Note that, since $\Ext^1_R(\syz^c_{ R} N^{\oplus b}, N^{\oplus a})$ is a finite length $R$-module, in the following the natural maps are isomorphisms:  $$\Ext^1_R(\syz^c_{ R} N^{\oplus b}, N^{\oplus a}) \cong \Ext^1_R(\syz^c_{ R} N^{\oplus b}, N^{\oplus a})\otimes_R\widehat{R} \cong  \Ext^1_{\widehat{R}}(\widehat{\syz^c_{ R} N^{\oplus b}}, \widehat {N^{\oplus a}})$$
As the composition of these two natural maps is given by the completion of short exact sequences, we deduce that there is a short exact sequence of $R$-modules $0\to N^{\oplus a} \to X \to \syz^c_{ R} N^{\oplus b}\to 0$, where $K\cong \widehat X$; see also \cite[2.7(i)]{TheBook}. 

Note that an $R$-module $T$ is locally free on the punctured spectrum of $R$ if and only if $\widehat{T}$ is locally free on the punctured spectrum of $\widehat{R}$. Therefore, since $K\cong \widehat X$ and $K$ is locally free on the punctured spectrum of $\widehat{R}$, we conclude that $X$ is locally free on punctured spectrum of $R$. As $\rHdim_{\widehat R}(\widehat{X})=\rHdim_{\widehat R}(K)=r-1$, we see by the induction hypothesis that $\rHdim_{R}(X)\le \rHdim_{\widehat R}(\widehat X)$. Hence, in view of Remark \ref{WH}(ii), the exact sequence $0\to N^{\oplus a} \to X \to \syz^c_R N^{\oplus b}\to 0$ implies that $\rHdim_R(N)\le \rHdim_R(X)+1\leq (r-1)+1=r$. This completes the induction and the proof.   
\end{proof}


\begin{cor}\label{corcompl} Let $R$ be a local ring and let $N$ be an $R$-module which is locally free on punctured spectrum of $R$. Then $\rpdim_{\widehat R}(\widehat N)=\rpdim_R(N)$ and $\rGdim_{\widehat R}(\widehat N)=\rGdim_R(N)$. 
\end{cor}  

\begin{proof} Note that $\pd_{\widehat R}(\widehat M)=\pd_R(M)$ and $\Gdim_{\widehat R}(\widehat M)=\Gdim_R(M)$ for each $R$-module $M$; see, for example, \ref{obs}(i) and \cite[5.10(c)]{Larsthesis}. Therefore the result follows from Proposition \ref{comple}.
\end{proof}

We are now ready to prove Theorem \ref{mainthm} as advertised in the introduction.

\begin{proof}[Proof of Theorem \ref{mainthm}] Note that have:
$$
\rpdim_{\widehat{S}}(k)\leq \rpdim_R( \Omega^n_{R} k)+n \leq \rpdim_R( k)+n <\infty.
$$
Here, in view of \ref{conc}, the first inequality is due to Proposition \ref{p2} and the second one is due to Corollary \ref{nzdsec1}; the third inequality is clear if $R$ is regular, and it holds by \ref{mincor} in case $R$ is singular.  Now, since $k$ is a complete $S$-module that is locally free on the punctured spectrum of $S$, it follows that $\rpdim_{S}(k)=\rpdim_{\widehat S}(\widehat k)=\rpdim_{\widehat S}(k)$; see from Corollary  \ref{corcompl}. This completes the proof.
\end{proof}  

We finish this section with a further application of Proposition \ref{p2} and Corollary \ref{nzdsec1}. Note that a local ring $R$ is called \emph{G-regular} \cite{TakG} provided that $\Gdim_R(M)<\infty$ if and only if $\pd_R(M)<\infty$ for each $R$-module $M$. As, in general, over a local ring $R$, we have that $\Gdim_R(M)=\pd_R(M)$ for each $R$-module $M$ with $\pd_R(M)<\infty$, it is clear that each regular local ring is G-regular. A nontrivial fact is that a non-Gorenstein Cohen-Macaulay local ring of minimal multiplicity is $G$-regular; see \cite[5.1]{TakG} for the details. On the other hand, deformations of G-regular rings, that is, G-regular rings modulo ideals generated by regular sequences, need not be G-regular as we see next:  

\begin{eg} \label{exendo} Let $R=\comp[\![t^3, t^4, t^5]\!]=\comp[\![x,y,z]\!]/(x^3-yz, z^2-yx^2, y^2-xz)$ and let $S=R/x^2R$. Then $x^2$ is a non zero-divisor on $R$ and $R$ is G-regular since it is a non-Gorenstein Cohen-Macaulay local ring with minimal multiplicity. Moreover, as $x^2 \in \fm^2$, it follows from \cite[4.6]{TakG} that $S$ is not G-regular. \qed
\end{eg}

Although deformations of G-regular rings are not G-regular in general, we finish this section by making a related observation: a deformation of a G-regular ring is \emph{reducing G-regular}, that is, a ring over which each module that has finite reducing Gorenstein dimension also has finite reducing projective dimension. First we note that each G-regular ring is reducing G-regular.

\begin{chunk} \label{lr} Let $R$ be a G-regular local ring. Then, by definition, $\pd_R(M)=\Gdim_R(M)$ for each $R$-module $M$. Therefore it follows from Definition \ref{rdim} that 
$\rpdim_R(M)=\rGdim_R(M)$ for each $R$-module $M$.
\end{chunk}

\begin{prop} \label{endo} Let $R$ be a G-regular local ring and let $S=R/\underline{x}R$, where $\underline{x}=x_1,\ldots, x_n \subseteq \fm$ is an $R$-regular sequence. If $M$ is a (finitely generated) $S$-module, then it follows that $$\rGdim_S(M)\leq  \rpdim_S(M) \leq  \rGdim_S(M)+n.$$  So, for each finitely generated $S$-module $N$, we have $\rpdim_S(N)<\infty$ if and only if $\rGdim_S(N)<\infty$. 
\end{prop}

\begin{proof} Let $M$ be a finitely generated $S$-module. Then the first inequality follows by Definition \ref{rdim}. For the second inequality, note that we have:
\begin{equation}\notag{}
\rpdim_{S}(M)\le \rpdim_R(\syz^n_RM)+n=\rGdim_R(\syz^n_R M)+n \leq \rGdim_S(M)+n.
\end{equation}
Here the first inequality, in view of \cite[1.3.5]{BH}, follows from Proposition \ref{p2}, the second inequality follows from Theorem \ref{nzdse} and Remark \ref{nzdseremark}, and the equality is due to \ref{lr}.
\end{proof}


\section{Proof of Proposition \ref{213}}






This section is devoted to a proof of Proposition \ref{213}. We start by noting that:

\begin{rem}\label{son} Let $R$ be a local ring, $M$ be an $R$-module, $F$ be a free $R$-module, and let $\Hdim$ be a homological invariant of $R$-modules. If $\{K_0,\dots,K_r\}$ is a reducing $\Hdim$-sequence of $M$, then one can find some free $R$-modules $G_1,G_2,\dots,G_r$ such that
$\{K_0, K_1\oplus G_1,\dots, K_r\oplus G_r\}$ is a reducing $\Hdim$-sequence of $M\oplus F$; see Definition \ref{rdim}. Therefore $\rHdim_R(M\oplus F)\le \rHdim_R(M)$. \qed
\end{rem}

Proposition \ref{213}, in view of Remark \ref{son}, is subsumed by the following result:

\begin{prop} \label{sonprop} Let $R$ be a local ring and let $\Hdim$ be a homological invariant of $R$-modules which satisfies the condition (iii) of \ref{setup}. Let $0\to F \to N \to M \to 0$ be an exact sequence of nonzero $R$-modules, where $F$ is free. Then it follows that $\rHdim_R(M)\le \rHdim_R(N)$.
\end{prop}

\begin{proof}
To prove the claim, we assume $\rHdim_R(N)\le r$ for some integer $r\geq 0$, and proceed by induction on $r$. Note, if $r=0$, then the claim holds due to the condition (iii) of \ref{setup}. So we assume $r\geq 1$. 


It follows that there is a short exact sequence of $R$-modules $0 \to N^{\oplus a} \xrightarrow[]{\chi}  K \to \syz_R^n N^{\oplus b} \to 0$, where $\rHdim_R(K)=\rHdim_R(N)-1\le r-1$ and $a,b\ge 1$, $n\ge 0$ are integers; see Remark \ref{WH}(ii). Now we consider the following commutative diagram (with exact rows and columns) obtained by taking the pushout of the maps $\chi$ and $\psi^{\oplus a}$, where $0\to F \to N \xrightarrow[]{\psi}  M \to 0$ is the exact sequence we consider.
\[
\begin{tikzcd}
& 0 \arrow[d] & 0 \arrow[d] & & \\
& F^{\oplus a} \arrow[d] \arrow[r, equal] & F^{\oplus a} \arrow[d] & & \\
0 \arrow[r] & N^{\oplus a} \arrow[d, "\psi^{\oplus a}"] \arrow[r, "\chi"] & K \arrow[r] \arrow[d] & \Omega_R^n N^{\oplus b} \arrow[r] \arrow[d, equal] & 0 \\
0 \arrow[r] & M^{\oplus a} \arrow[d] \arrow[r] & L \arrow[r] \arrow[d]  & \Omega_R^n N^{\oplus b} \arrow[r] & 0 \\
& 0 & 0 & &  
\end{tikzcd} 
\]
Now, since $\rHdim_R(K)\le r-1$, we use the induction hypothesis on $r$ with the short exact sequence $0\to F^{\oplus a} \to K \to L \to 0$ obtained in the above diagram, and conclude that 
\begin{equation}\tag{\ref{sonprop}.1}
\rHdim_R(L)\le \rHdim_R(K)=\rHdim_R(N)-1.
\end{equation}
\emph{Case 1}: Suppose $n\geq 1$ and consider the exact sequence $0\to M^{\oplus a} \to L \to \syz_R^n N^{\oplus b} \to 0$. 

We take syzygies of the short exact sequence $0\to F \to N \to M \to 0$, and see that, for each $v\geq 1$, there exists a free $R$-module $H_v$ such that $\syz_R^v M \cong \syz_R^v N \oplus H_v$. So, by adding free summands to the exact sequence $0\to M^{\oplus a} \to L \to \syz_R^n N^{\oplus b} \to 0$, we obtain the exact sequence: 
\begin{equation}\tag{\ref{sonprop}.2}
0\to M^{\oplus a} \to L \oplus H_n^{\oplus b} \to \syz_R^n N^{\oplus b} \oplus{H_n^{\oplus b}} \to 0. 
\end{equation}
The sequence (\ref{sonprop}.2), in view of the isomorphism $\syz_R^n M \cong \syz_R^n N \oplus H_n$, gives the exact sequence: 
\begin{equation}\tag{\ref{sonprop}.3}
0\to M^{\oplus a} \to L \oplus H_n^{\oplus b} \to \syz_R^n M^{\oplus b} \to 0.
\end{equation}
Now we use (\ref{sonprop}.3) and conclude from Remarks \ref{WH}(ii) and \ref{son} that:
\begin{equation}\tag{\ref{sonprop}.4}
\rHdim_R(M)\le \rHdim_R(L \oplus H_n^{\oplus b} )+1 \leq \rHdim_R(L)+1. 
\end{equation}
Thus, since $\rHdim_R(K)+1=\rHdim_R(N)$, the proof is complete in view of (\ref{sonprop}.1) and (\ref{sonprop}.4).

\emph{Case 2}: Suppose $n=0$ and consider the exact sequence $0\to M^{\oplus a} \to L \to \syz_R^n N^{\oplus b} \to 0$. 

Recall that, by the hypothesis and the previous pushout diagram, we have the short exact sequences $0\to M^{\oplus a} \to L \to N^{\oplus b} \to 0$ and $0\to F^{\oplus b} \to N^{\oplus b} \xrightarrow[]{\psi^{\oplus b}}  M^{\oplus b} \to 0$. Next we take the pullback of the maps $L \to N^{\oplus b}$ and $\psi^{\oplus b}$, and obtain the following commutative diagram with exact rows and columns:
\[
\begin{tikzcd}
 & & 0 \ar[d] & 0 \ar[d] & \\
0 \ar[r] & M^{\oplus a} \ar[r] \ar[d,equal] & Y \ar[r] \ar[d] & F^{\oplus b} \ar[r] \ar[d, "\psi^{\oplus b}" ] & 0 \\
0 \ar[r] & M^{\oplus a} \ar[r] & L \ar[r] \ar[d] & N^{\oplus b} \ar[r] \ar[d] & 0\\
 & & M^{\oplus b} \ar[r,equal] \ar[d] & M^{\oplus b} \ar[d] &\\
 & & 0 & 0 &
\end{tikzcd}
\]
The top exact sequence in the above diagram shows that $Y\cong M^{\oplus a}\oplus F^{\oplus b}$ as $F^{\oplus b}$ is a free $R$-module. So the left vertical exact sequence in the above diagram gives the short exact sequence: 
\begin{equation}\tag{\ref{sonprop}.5}
0 \to M^{\oplus a}\oplus F^{\oplus b} \to L \to M^{\oplus b} \to 0. 
\end{equation}
We add, if needed, free $R$-modules to (\ref{sonprop}.5) and obtain an exact sequence for some integers $c, d\geq 0$: 
\begin{equation}\tag{\ref{sonprop}.6}
0 \to (M \oplus F^{\oplus c})^{\oplus a} \to L \oplus F^{\oplus d} \to (M \oplus F^{\oplus c})^{\oplus b} \to 0.
\end{equation}

Next we observe the following inequalities:
\begin{align}\notag{}
\rHdim(M) \le  \rHdim_R(M\oplus F^{\oplus c}) \leq \rHdim_R(L \oplus F^{\oplus d})+1 \le & \rHdim(L) + 1 \\ \notag{} \le & \rHdim(K)+1 \\ \notag{} = & \rHdim(N) \\ \notag{} \leq & r.
\end{align}
Here we obtain the second inequality by applying Remark \ref{WH}(ii) to (\ref{sonprop}.6), the third inequality is due to Remark \ref{son}, the fourth inequality and the equality follow from (\ref{sonprop}.1), the last inequality is our assumption, and the first inequality can be deduced by the next claim:

\begin{Claim} Let $A$ and $H$ be $R$-modules such that $H$ is free and $\rHdim_R(A\oplus H)\leq r$. Then it follows $\rHdim_R(A)\le \rHdim_R(A\oplus H)$.
\end{Claim}

Proof of the Claim: We note, due to  Definition \ref{rdim} and Remark \ref{WH}(ii), that there exists a short exact sequence of $R$-modules $0 \to (A\oplus H)^{\oplus a} \to K \xrightarrow[]{p} \syz_R^{s}(A\oplus H)^{\oplus b} \to 0$, where $a,b\ge 1$, $s\ge 0$ are integers and $\rHdim_R(K)=\rHdim_R(A\oplus H)-1\le r-1$. Note also that $\syz_R^s(A \oplus H)^{\oplus b}\cong \syz_R^sA^{ \oplus b} \oplus G$ for some free $R$-module $G$. Thus we have an exact sequence of the form:
\begin{equation}\tag{\ref{sonprop}.7}
0 \to (A\oplus H)^{\oplus a} \to K \to \syz_R^sA^{ \oplus b} \oplus G \to 0.
\end{equation}
Now we use \ref{c26} with the sequence (\ref{sonprop}.7) and obtain the following exact sequence of $R$-modules:
\begin{equation}\tag{\ref{sonprop}.8}
0 \to (A \oplus H)^{ \oplus a} \xrightarrow[]{\alpha} K' \to \syz_R^{n}A^{\oplus b} \to 0, 
\end{equation}
where $K\cong K'\oplus G$. 

Next we consider the commutative diagram (with exact rows and columns) that is obtained by taking the pushout of the map $\alpha$ from (\ref{sonprop}.8) with the canonical surjection $(A\oplus H)^{\oplus a} \twoheadrightarrow A^{\oplus a}$:
\[
\begin{tikzcd}
& 0 \ar[d] & 0 \ar[d] & &\\
& H^{\oplus a} \ar[d] \ar[r,equal] & H^{\oplus a} \ar[d] & &\\
0 \ar[r] & (A\oplus H)^{\oplus a} \ar[r,"\alpha"] \ar[d] & K' \ar[r] \ar[d] & \syz_R^{n}A^{\oplus b} \ar[r] \ar[d,equal] & 0\\
0 \ar[r] & A^{\oplus a} \ar[r] \ar[d] & X \ar[r] \ar[d] & \syz_R^nA^{\oplus b} \ar[r] & 0\\
& 0 & 0 & &
\end{tikzcd}
\]
Note that, since $K\cong K'\oplus G$, we have an exact sequence $0 \to G \to K \to K' \to 0$. We use this exact sequence and, since $\rHdim_R(K)\le r-1$, conclude by the induction hypothesis on $r$ that 
\begin{equation}\tag{\ref{sonprop}.9}
\rHdim_R(K') \leq \rHdim_R(K) \leq r-1. 
\end{equation}

Next we consider the exact sequence obtained from the above diagram: $0 \to H^{\oplus a} \to K' \to X \to 0$. As $\rHdim_R(K')\le r-1$ by (\ref{sonprop}.9), we make use of the induction hypothesis on $r$ and conclude that
\begin{equation}\tag{\ref{sonprop}.10}
\rHdim_R(X)\le \rHdim_R(K') \leq \rHdim_R(K).
\end{equation}

We also consider the exact sequence obtained from the above diagram: $0 \to A^{\oplus a} \to X \to \syz_R^n A^{\oplus b} \to 0$. We deduce from this short exact sequence and Remark \ref{WH}(ii) that:
\begin{equation}\tag{\ref{sonprop}.11}
\rHdim_R(A) -1\le \rHdim_R(X).
\end{equation}
Finally, as $\rHdim_R(K)=\rHdim_R(A\oplus H)-1$, we see $\rHdim_R(A) \le \rHdim_R(A\oplus H)$ by (\ref{sonprop}.10) and (\ref{sonprop}.11).
This completes the proof of the Claim, and hence that of the proposition.
\end{proof}

\section*{acknowlwdgements} The authors are grateful to Roger Wiegand for discussions, and for his helpful and valuable comments and suggestions on the manuscript.

\bibliography{a}
\bibliographystyle{amsplain}
\end{document}